\documentclass[12pt,leqno,a4paper]{amsart}
\usepackage{amssymb,enumerate}
\usepackage[OT2, T1]{fontenc}
\usepackage[russian, english]{babel}
\newcommand{\FF}{{\mathbb{F}}}

\newcommand{\fA}{{\mathfrak{A}}}
\newcommand{\fS}{{\mathfrak{S}}}

\newcommand{\bC}{{\mathbf{C}}}
\newcommand{\bH}{{\mathbf{H}}}
\newcommand{\bT}{{\mathbf{T}}}

\newcommand{\cO}{{\mathcal{O}}}
\newcommand{\cR}{{\mathcal{R}}}
\newcommand{\cS}{{\mathcal{S}}}

\newcommand{\Aut}{{\operatorname{Aut}}}
\newcommand{\fix}{\operatorname{fix}}                   
\newcommand{\Out}{{\operatorname{Out}}}
\newcommand{\Syl}{\operatorname{Syl}}                   
\newcommand{\Stab}{\operatorname{Stab}}
\newcommand{\GL}{{\operatorname{GL}}}
\newcommand{\PGL}{{\operatorname{PGL}}}
\newcommand{\PSL}{{\operatorname{L}}}
\newcommand{\SL}{{\operatorname{SL}}}
\newcommand{\PGU}{{\operatorname{PGU}}}
\newcommand{\PSU}{{\operatorname{U}}}
\newcommand{\SU}{{\operatorname{SU}}}
\newcommand{\OO}{{\operatorname{O}}}
\newcommand{\PSp}{{\operatorname{S}}}
\newcommand{\Sp}{{\operatorname{Sp}}}
\newcommand{\tw}[1]{{}^#1\!}
\newcommand\co{{:}}

\newcommand{\seq}[1]{\left<#1\right>}                 
\newcommand{\set}[1]{\left\{#1\right\}}               
\newcommand{\NN}{\mbox{\foreignlanguage{russian}{\textrm{I}}}}    
\newcommand{\udot}{^{{}^{\textbf{ .}}}}               
\newcommand{\su}{\lhd\!\!\lhd \, }                    

\newtheorem{thm}{Theorem}[section]
\newtheorem{lem}[thm]{Lemma}

\newtheorem{prop}[thm]{Proposition}

\newtheorem*{thmW}{Theorem}
\newtheorem*{thmA}{Theorem A}
\newtheorem*{corB}{Corollary B}
\newtheorem*{thmB}{Theorem C}

\theoremstyle{definition}
\newtheorem{exmp}[thm]{Example}

\theoremstyle{remark}

\begin{document}

\title{A generalisation of a theorem of Wielandt}

\date{\today}

\author{Francesco Fumagalli}
\address{Dipartimento di Matematica e Informatica ``Ulisse Dini'',
         viale Morgagni 67/A, 50134 Firenze, Italy.}
\email{fumagalli@math.unifi.it}
\author{Gunter Malle}
\address{FB Mathematik, TU Kaiserslautern, Postfach 3049,
         67653 Kaisers\-lautern, Germany.}
\email{malle@mathematik.uni-kl.de}

\thanks{The second author gratefully acknowledges financial support by ERC
  Advanced Grant 291512.}



\begin{abstract}
In 1974, Helmut Wielandt proved that in a finite group $G$, a subgroup $A$ is
subnormal if and only if it is subnormal in every $\seq{A,g}$ for all $g\in G$.
In this paper, we prove that the subnormality of an odd order nilpotent
subgroup $A$ of $G$ is already guaranteed by a seemingly weaker condition:
$A$ is subnormal in $G$ if for every conjugacy class $C$ of $G$ there exists
$c\in C$ for which $A$ is subnormal in $\seq{A,c}$.
We also prove the following property of finite non-abelian simple groups:
if $A$ is a subgroup of odd prime order $p$ in a finite almost simple group $G$,
then there exists a cyclic $p'$-subgroup of $F^*(G)$ which does not normalise
any non-trivial $p$-subgroup of $G$ that is generated by conjugates of~$A$.
\end{abstract}

\maketitle


\section{Introduction}
The main result of our paper is the following criterion for the existence
of a non-trivial normal $p$-subgroup in a finite group:

\begin{thmA}   \label{Main_Thm}
 Let $G$ be a finite group and $p$ be an odd prime. Let $A$ be a $p$-subgroup
 of $G$ such that
 $$\text{for every conjugacy class $C$ of $G$ there exists $g\in C$
  with $A$ subnormal in $\seq{A,g}$}.\leqno{(*)}$$
 Then $A\leq O_p(G)$.
\end{thmA}
As an immediate consequence we have that

\begin{corB}    \label{Main_Cor}
 If $A$ is an odd order nilpotent subgroup of a finite group $G$ satisfying
 condition $(*)$, then $A$ is subnormal in $G$.
\end{corB}

This can be considered a generalisation of the following result due to
H.~Wielandt (see \cite[7.3.3]{LS}):

\begin{thmW}[Wielandt]
 Let $A$ be a subgroup of a finite group $G$. Then the following conditions
 are equivalent.
 \begin{enumerate}
  \item[\rm(i)] $A$ is subnormal in $G$;
  \item[\rm(ii)] $A$ is subnormal in $\seq{A,g}$ for all $g\in G$;
  \item[\rm(iii)] $A$ is subnormal in $\seq{A,A^g}$ for all $g\in G$;
  \item[\rm(iv)] $A$ is subnormal in $\seq{A,A^{a^g}}$ for all $a\in A$,
   $g\in G$.
 \end{enumerate}
\end{thmW}

Our proof of Theorem~A makes use of a reduction argument to arrive at a
question about finite almost simple groups and then prove a property of these
groups which may be of independent interest:

\begin{thmB}   \label{thm:simples}
 Let $G$ be a finite almost simple group with simple socle $S$ and $p>2$ be a
 prime dividing $|G|$. Let $A\le G$ be cyclic of order~$p$. Then there exists
 a cyclic $p'$-subgroup $X\le S$ such that
 $$\NN_G^A(X,p)=\varnothing.$$
\end{thmB}

Here, $\NN_G^A(X,p)$ denotes the set of non-trivial $p$-subgroups of $G$
generated by conjugates of $A$ and normalised by $X$.

Our proof is therefore related to (and relies on) the classification of
finite simple groups. It should be noted that for $p=2$ the conclusions of
Theorem~A and Theorem~C are no longer true. In particular condition $(\ast)$
does not imply that $A\leq O_2(G)$. An easy example is reported at the end of
Section~\ref{sec:main}.

In Section~\ref{sec:almost simple} we give, after some preparations, the
proof of Theorem~C, and then in Section~\ref{sec:main} show the reduction of
Theorem~A to the case of almost simple group.\\
We end with Section~\ref{sec:others}, where we analyse similar variations
related to the other criteria for subnormality given by the original Theorem
of Wielandt, namely conditions (iii), better known as the Baer-Suzuki Theorem.
We show that in general these generalisations fail to guarantee the
subnormality of odd $p$-subgroups. For other variations on the Baer--Suzuki
Theorem the interested reader may consult \cite{X}, \cite{GR}, \cite{G1},
\cite{G2}, \cite{GM1} and \cite{GM2}.

\section{Almost simple groups}   \label{sec:almost simple}

\subsection{Notation and preliminary results}

In this section we let $S$ be a non-abelian finite simple group and $G$ any
group such that $S\leq G\leq \Aut(S)$. For $p$ a prime divisor of $|G|$
denote by $\cS_p(G)$ the set of all (possibly trivial) $p$-subgroups of $G$.
For a $p'$-subgroup $X$ of $S$ we denote by $\NN_G(X,p)$ the set of
$p$-subgroups of $G$ normalised by $X$, namely
$$\NN_G(X,p)=\set{Y\in\cS_p(G)\mid X\leq N_G(Y)}.$$
Also for $A\in\cS_p(G)$ set
$$\NN_G^A(X,p):=\set{Y \in \NN_G(X,p)\mid Y \textrm{ is generated by
  $G$-conjugates of } A}.$$
Note that if $A\leq S$, then $\NN_G^A(X,p)\subseteq \NN_S(X,p)$, otherwise
if $A\not\leq S$ then no $E\in \NN_G^A(X,p)$ lies in $S$.


We aim to prove Theorem~C, which we restate:

\begin{thm}   \label{thm:almost-simple}
 Let $G$ be a finite almost simple group with simple socle $S$. Then with the
 same notation as above, for every odd prime $p$ dividing $|G|$ and every
 $A\leq G$ of order $p$, there exists a cyclic $p'$-subgroup $X\leq S$ such
 that $\NN_G^A(X,p)=\varnothing$.
\end{thm}

In \cite{AK}, a similar condition is considered. The authors investigate finite
groups $G$ and primes $p$ that have the following property:
\begin{center}
(R2)\  all nilpotent hyperelementary $p'$-subgroups $X$ of $F^*(G)$
satisfy $\NN_G(X,p)\ne 1$
\end{center}
where a \emph{hyperelementary group} $X$ is one for which $O^q(X)$ is cyclic,
for some prime $q$; basically a nilpotent hyperelementary $p'$-group $X$ is a
direct product of a Sylow $q$-subgroup for some prime $q\ne p$, and a cyclic
$p'$-group. They show (\cite[Thm.~2]{AK}) that the only almost simple
group $G$ satisfying (R2) at a prime $p$ is for $S=\PSL_3(4)$, $p=2$,
and $4$ dividing $|G:S|$.

Note that assumption (R2) implies our assumption. See also \cite{CC} for an
analogous condition and its related subnormality criteria.

\begin{lem}   \label{lem:out}
 In the situation of Theorem~\ref{thm:almost-simple} assume that $A\not\le S$.
 Let $X$ be a non-trivial $p'$-subgroup of $S$ and $E\in \NN^A_G(X,p)$.
 Then $X$ commutes with some non-trivial $p$-element in $G\setminus S$.
\end{lem}

\begin{proof}
By the coprime action of $X$ on $E$, we have that $E=[E,X]C_E(X)$. As $E$ is
generated by conjugates of $A$ and $A\not\leq S$, we necessarily have that
$C_E(X)\not\leq S$.
\end{proof}

\begin{lem}   \label{lem:normElem}
 In the situation of Theorem~\ref{thm:almost-simple} if
 $\NN_G^A(X,p)\ne\varnothing$ then $X$ normalises a non-trivial elementary
 abelian $p$-subgroup of $S$ or it centralises a non-trivial $p$-element of~$G$.
\end{lem}

\begin{proof}
Let $E$ be a non-trivial $p$-subgroup of $G$ normalised by $X$. If $X$ does
not centralise any non-trivial $p$-element, then $A\le S$ by
Lemma~\ref{lem:out} and hence $E\le S$. Now $X$ also normalises $Z(E)$, and
then also $\Omega_1(Z(E))$, the largest elementary abelian subgroup of
$Z(E)$.
\end{proof}

The following is a well-known consequence of the classification of finite
simple groups and can be found for example in \cite[2.5.12]{GLS}.

\begin{prop}   \label{prop:oddout}
 Let $S$ be non-abelian simple, $S<G\le\Aut(S)$ and assume that
 $x\in G\setminus S$ has odd prime order $p$. Then $S$ is of Lie type and one
 of the following occurs:
 \begin{enumerate}
  \item[\rm(1)] $x$ is a field automorphism of $S$;
  \item[\rm(2)] $x$ is a diagonal automorphism of $S$ and one of
  \begin{enumerate}
   \item[\rm(2.1)] $S=\PSL_n(q)$ with $p|(n,q-1)$,
   \item[\rm(2.2)] $S=\PSU_n(q)$ with $p|(n,q+1)$,
   \item[\rm(2.3)] $p=3$, $S=E_6(q)$ with $3|(q-1)$,
   \item[\rm(2.4)] $p=3$, $S=\tw2E_6(q)$ with $3|(q+1)$; or
  \end{enumerate}
  \item[\rm(3)] $p=3$ and $x$ is a graph or graph-field automorphism of
   $S=\OO_8^+(q)$.
 \end{enumerate}
\end{prop}

We prove Theorem~\ref{thm:almost-simple} by treating separately the
cases: $S$ is alternating, sporadic or a simple group of Lie type.

\subsection{The case of alternating groups.}
Throughout the rest of this subsection we assume $S=\fA_n$, with $n\geq 5$ and
$G$ such that $S\leq G\leq \Aut(S)$. For $X$ a cyclic $p'$-subgroup of
$S$ we denote by $E$ any element of $\NN_G(X,p)$. Also, we tacitly
assume that any such $E$ is elementary abelian (see Lemma~\ref{lem:normElem}).

The following elementary result \cite[Lemma~3]{AK} will be used several times.

\begin{lem}   \label{lem:Sym}
 Let $X\leq G\leq \fS_n$ and  $E\in \NN_G(X,p)$. If $E$ acts non-trivially on
 some $X$-orbit $\cO$, then $p$ divides $|\cO|$.
\end{lem}

\begin{proof}
As $X$ acts on $\fix_{\cO}(E)$, $\fix_{\cO}(E)=\varnothing$ and
$|\cO|\equiv |\fix_{\cO}(E)|\equiv 0$ (mod $p$).
\end{proof}

\begin{prop}   \label{prop:Alt}
 If $\NN_G(X,p)\ne1$ for every cyclic $p'$-subgroup $X$ of $S$, then $S=\fA_6$
 and $(G,p)\in\set{(\PGL_2(9),2),(\Aut(\fA_6),2)}$.
 In particular for every odd prime $p$ and every subgroup $A$ of order $p$
 there exists a cyclic $p'$-subgroup $X$ for which $\NN^A_G(X,p)=\varnothing$.
\end{prop}

\begin{proof}
The cases $n=5$ and $n=6$, with $G\leq \fS_6$, follow quite immediately by
taking, as subgroup $X$ a Sylow $5$-subgroup, if $p=2$ or $p=3$, and a cyclic
subgroup of order $3$ if $p=5$. Assume that $n=6$ and $G\not\leq \fS_6$.
As $|G:\fA_6|$
is either $2$ or $4$, if $p$ is odd then $\NN_G(X,p)=\NN_{\fA_6}(X,p)$,
for every $X\leq \fA_6$. Therefore, by what we have just proved, there is some
cyclic $p'$-subgroup $X$ of $\fA_6$ for which $\NN_G(X,p)=1$. Let $p=2$.
The group $G=M_{10}$ contains no elements of order 10 (see \cite{Atlas}).
We take $X$ a cyclic subgroup of order $5$ of $G$. Note that
if $E$ is any $2$-subgroup of $G$ normalised by $X$, then $E=[E,X]$ and so
$E$ lies in $\fA_6$, but then $\NN_G(X,2)=\NN_{\fA_6}(X,2)=1$. Let now
$G=\PGL_2(9)$ or $G=\Aut(\fA_6)$. Then every element of odd order normalises a
non-trivial $2$-subgroup of $G$, basically the elements of order~$3$ normalise
a copy of $\fA_4$ lying in $\fA_6$, while the elements of order $5$ centralise
always an outer involution (see \cite{Atlas}). Therefore we have that
$\NN_G(X,2)\ne 1$ for every cyclic odd order subgroup $X$ of $G$.

\noindent
Assume for the rest of the proof that $n\geq 7$ and argue by contradiction.
We treat separately the two cases: 1) $n$ is even and 2) $n$ is odd.

\smallskip\noindent
{\bf Case 1}. $n$ is even.

\noindent
Assume first that $p\vert (n-1)$.
In particular $p$ is odd. We take as cyclic $p'$-subgroup $X$ of $\fA_n$
the one generated by $x=(12)(3\ldots n)$. Let $E$ be a non-trivial element of
$\NN_G(X,p)$. Now the set $\set{1,2}$ cannot lie in $\fix(E)$, otherwise $E$
acts non-trivially on $\set{3,\ldots,n}$, and by Lemma \ref{lem:Sym} we
would have that $p$ divides $n-2$, which is not the case being a divisor
of $n-1$. Also from the fact that $E=E^x$, it follows that none of $1$
and $2$ are fixed by $E$. But then, as $p>2$, $EX$ is transitive on
$\set{1,2,\ldots,n}$, and since $p$ does not divide $n$
we have a contradiction.

Assume now that $p\nmid (n-1)$.\\
We choose first $X=\seq{x}$ with $x=(2\ldots n)$ and let $1\ne E\in\NN_G(X,p)$.
By Lemma \ref{lem:Sym} we have that $E$ does not fix $1$.
Then $EX$ is transitive on $\set{1,2,\ldots,n}$. In particular we have that $n$
is a power of $p$ and since it is even $p=2$. Say $n=2^r\geq 8$.\\
We can now change our testing subgroup $X=\seq{x}$ and choose
now $x=(123)(4\ldots n)$. This is  of course a cyclic $p'$-subgroup of $\fA_n$.
Let $E$ be a non-trivial elementary abelian $2$-subgroup lying in $\NN_G(X,2)$.
Note that $E$ acts fixed-point-freely on $\set{1,2,\ldots,n}$. Indeed if $E$
fixes a point, then as $X$ normalises $E$, we have that either $\set{1,2,3}$ or
$\set{4,\ldots, n}$ lie in $\fix(E)$. In any case we reach a contradiction with
Lemma \ref{lem:Sym}. Now we claim that $E$ is transitive. Let $\cO$ be the
$E$-orbit containing $1$. Then $\cO$ cannot be contained in $\set{1,2,3}$,
otherwise $E$ being a $2$-group, there will be a fixed point of $E$ in
$\set{1,2,3}$, which is not the case.
Let therefore $e\in E$ be such that $1e=i\in\set{4,\ldots,n}$.
The subgroup $\seq{x^3}$ is transitive on $\set{4,\ldots,n}$, as $3$ is coprime
to $n-3=2^r-3$, thus for every $j\in\set{4,\ldots,n}$ we may take some
$y\in\seq{x^3}$ such that $iy=j$. But then
$$1(e^y)=1(y^{-1}ey)=1(ey)=i(y)=j$$
and since $e^y\in E$ the element $j$ lies in $\cO$. In particular we have
proved that $\set{4,\ldots,n}\subseteq \cO$ and, since $n-2=2^r-2$ is not a
power of 2 as $n\geq 8$, we have that $\cO=\set{1,2,\ldots, n}$ and $E$ acts
regularly on it. Now let $e$ be the unique element of $E$ that maps $1$
to $2$, then $e^x$ maps $2$ to $3$. Now as $[e,e^x]=1$ we have that
$$1(ee^x)=2(e^x)=3=1(e^xe)$$
which means that $1(e^x)=3(e^{-1})$, and so $1e^x\not\in \set{1,2,3}$.
If we set $1e^x=j$, for some $j\in\set{4,\ldots,n}$, we reach a contradiction,
since
$$e=(12)(3,j)\ldots \quad {\textrm{and}} \quad e^x=(23)(1,j)\ldots $$
but also as $n\geq 8$, $jx\ne j$ and so
$e^x=(1x,2x)(3x,jx)\ldots=(23)(1,jx)\ldots$.

\smallskip\noindent
{\bf Case 2}. $n$ is odd.

\noindent
In this situation we have that $p\vert n$. Indeed if this is not the case,
we take $X$ the subgroup generated by an $n$-cycle of $\fA_n$. Now any
$1\ne E\in\NN_G(X,p)$ acts non-trivially on $\set{1,2,\ldots,n}$, and
therefore we reach a contradiction to Lemma~\ref{lem:Sym}. Thus $p\vert n$;
in particular $p$ is odd.\\
We take now $x$ the $(n-2)$-cycle $(3,4,\ldots, n)$ and $X=\seq{x}$.
Then any $1\ne E\in\NN_G(X,p)$ does not fix both $1$ and
$2$, otherwise by Lemma \ref{lem:Sym}, $p\vert (n-2)$ which is not the case as
$p\vert n$ and $p$ is odd. Assume that $1$ is not fixed by $E$ (otherwise
argue considering $2$ in place of $1$) and let $\cO_1$ be the $E$-orbit
containing $1$. Since $p$ is odd there is some $e\in E$ such that $1e=i$ for
some $i\in\set{3,\ldots,n}$. Now as $X$ is transitive on  $\set{3,\ldots,n}$,
for every $j\in\set{3,\ldots,n}$ there is some power $m$ of $x$ such
$ix^m=j$. But then
$$1(e^{x^m})=1(x^{-m}ex^m)=1(ex^m)=i(x^m)=j,$$
and as $e^{x^m}\in E$ we have proved that $\set{3,\ldots, n}\subseteq\cO_1$.
Since $p\nmid n-1$ we conclude that $E$ is transitive on
$\set{1,\ldots,n}$. We show now that $E$ is regular. The stabiliser
$\Stab_E(1)$ is normalised by $X$, and therefore if this is non-trivial, then
by Lemma~\ref{lem:Sym} we reach the contradiction $p\vert (n-2)$. It follows
that $E$ is regular on $\set{1,2,\ldots,n}$ and so $n=|E|=p^r$, and any
non-trivial element $\sigma$ of $E$ is a product of exactly $p^{r-1}$ cycles
of length $p$. In particular there exists a unique $\sigma\in E $ which
maps $1$ to $2$. We write
$$\sigma=\sigma_1\sigma_2\cdots \sigma_{p^{r-1}}$$
with $\sigma_1=(12u_3\ldots u_p)$, a $p$-cycle. Since $\sigma^x\in E$
maps $1$ to $2$, we necessarily have that $\sigma=\sigma^x$, but this is not
the case as $\sigma_1^x=(12u_3x\ldots u_px)\ne (12u_3\ldots u_p)=\sigma_1$.
\end{proof}

%
%
\subsection{The case of sporadic groups.}   \label{subsec:spor}
We assume now that $S$ is one of the 27 sporadic simple groups (including the
Tits simple group $\tw2F_4(2)'$), and $S\leq G\leq \Aut(S)$. As before $X$
will denote a cyclic $p'$-subgroup of $S$ and $E$ a non-trivial elementary
abelian $p$-subgroup of $G$ normalised by $X$. Our basic reference for
properties of sporadic groups is \cite{Atlas}.

\begin{prop}   \label{prop:Spor}
 Let $S$ be a simple sporadic group, $S\leq G\leq \Aut(S)$ and $p$ a prime.
 Then there exists a cyclic $p'$-subgroup $X$ of $S$ such that $\NN_G(X,p)=1$.
 In particular, for every odd prime $p$ and every subgroup $A$ of $G$ of
 order $p$, there exists a cyclic $p'$-subgroup $X$ such that
 $\NN^A_G(X,p)=\varnothing$.
\end{prop}

\begin{proof}
We extend a little our notation. Given a prime $p$ and a positive integer $q$
coprime to $p$, we write $\NN_G(q,p)$ for the set of $p$-subgroups of $G$
that are normalised by some cyclic $q$-subgroup of $S$.

Table~\ref{tab:spor} summarises the situation for the sporadic groups and
their automorphism groups. For every group $S$, we list a pair $(q;r)$ of
primes such that $\NN_G(q,p)=1$ for all $p\ne q$, and $\NN_G(r,q)=1$.
For four groups, $\NN_G(q,p)\ne1$ for another prime $p\ne q$, in which case
either $\NN_G(r,p)=1$, or we give a further integer $s$ such that
$\NN_G(s,p)=1$. Our choice of $(q;r)$, respectively $(q;r;s)$ works for both
$S$ and $\Aut(S)$.

\noindent
\begin{table}[htbp]
\caption{The case of sporadic groups.} \label{tab:spor}
$\begin{array}{|lll|lll|ll|}\hline
M_{11} & (11; 3)  &&  Co_{3} &  (23; 7)  &&  B      &  (47; 31) \\
M_{12} & (11; 3)  &&  Co_{2} &  (23; 5)  &&  M        & (59; 71)  \\
M_{22} & (11; 3)  &&  Co_{1}& (23\,(p\ne 2);33\,(p=2);13)&& J_1 &  (19; 11) \\
M_{23} & (23; 5)  &&  He     &  (17; 7)  &&  O'N      &  (31; 19) \\
M_{24} & (23; 5)  &&  Fi_{22}&  (13\,(p\ne 3); 11)&&  J_3    &  (19; 17) \\
J_{2}  & (7; 5)   &&  Fi_{23}&  (23\,(p\ne 2); 17)&&  Ly       &  (67; 37) \\
Suz    & (13; 11) &&  Fi_{24}'&  (29; 23)&&  Ru       &  (29; 13) \\
HS     & (11; 3)  &&  HN     &  (19; 11) &&  J_4      &  (43; 37) \\
M^cL   & (11\,(p\ne 2); 7) &&  Th& (31\,(p\ne 2); 19) && \tw2F_4(2)' & (13; 5) \\
     \hline
\end{array}$\\
\end{table}

\noindent
We prove the validity of Table 1 by considering the individual groups in turn.

The groups $S=M_{11}$, $J_1$, $J_2$, $M_{23}$, $M_{24}$, $Co_3$, $Co_2$, $Ru$,
$Ly$, $J_4$, $Fi_{23}$ have trivial outer automorphism group. The validity
of our claim is immediate from the known lists of maximal subgroups
\cite{Atlas}.

For $S=M_{12}$, $M_{22}$, $HS$, $He$, $J_3$, $O'N$, $HN$, $Th$, $\tw2F_4(2)'$
we have $|\Out(S)|=2$. For $G=S$ we can argue as before, while for $G=\Aut(S)$
we invoke Lemma~\ref{lem:out} for a suitable subgroup $X\le S$ of prime order
as listed in Table~\ref{tab:spor}.
We deal with the remaining groups in some more detail.

$S=Suz$. Here $|\Out(S)|=2$.
The maximal subgroups of $S$ of order divisible by 13 are isomorphic to
$G_2(4)$, $\PSL_3(3)\co2$ or $\PSL_2(25)$. As these have no elements of
order~11, we immediately obtain $\NN_{Suz}(5,11)=1$. Moreover, for any of
these groups, a Sylow $13$-subgroup does not normalise any other $p$-subgroup,
for $p\ne 13$. Thus $\NN_{Suz}(13,p)=1$ for every prime $p\ne 13$. Finally
the outer involutions do not centralise any element of order $13$, forcing
the same conclusions for $\Aut(Suz)$.

$S=M^cL$. Here $|\Out(S)|=2$.
The maximal subgroups of $S$ of order divisible by 11 are isomorphic to
$M_{11}$ and $M_{12}$. Therefore $\NN_{M^cL}(7,11)=\NN_{M^cL}(11,p)=1$ for
every $p\ne 11$. Now consider $G=\Aut(M^cL)$. Again, $X$ of order~11 shows that
there are no examples except possibly when $p=2$. In the latter case for $X$ we
take a cyclic subgroup of order~7. The Atlas \cite{Atlas} shows that
$C_G(X)\le S$, and so there can be no example for $G$ by Lemma~\ref{lem:out}.

$S=Co_1$. Here $\Out(S)=1$.
The maximal subgroups of $S$ of order divisible by 23  are isomorphic to
$Co_2$, $2^{11}\co M_{24}$ or $Co_3$. Since these groups have no elements of
order 13, we obtain that $\NN_{Co_1}(13,23)=1$. Moreover if $Y$ is any of
these maximal subgroups $\NN_{Y}(23,p)=1$, for every $p$ different from~23
and~2. Finally $Co_1$ has elements of order $33$ which are auto-centralising.
As $Co_1$ and $Co_2$ do not contain elements of order $33$, the unique maximal
subgroups of $S$ that have such an element are: $\PSU_6(2)\co\fS_3$,
$3^6\co2M_{12}$ and $3^{\cdot}Suz\co2$. Now, a cyclic subgroup of order $33$ in
$\PSU_6(2)\co\fS_3$ does not lie completely in $\PSU_6(2)$; therefore, if such
a subgroup normalises a non-trivial $2$-subgroup, then, since $S$ has no
elements of order $66$, we should have that $\NN_{\PSU_6(2)}(11,2)\ne 1$.
This is not the case as in $\PSU_6(2)$ the maximal subgroups of order
divisible by $11$ are $M_{22}$ and $\PSU_5(2)$. Consider now a cyclic subgroup
$Y$ of order $33$ inside $3^6\co2M_{12}$. This is the direct product of a
subgroup of order $3$ in $3^6$ by a Sylow $11$-subgroup of $2M_{12}$.
Assume that $X$ is a non-trivial $2$-subgroup of $3^6\co2M_{12}$ normalised by
$Y$. Then $X\in \NN_{2M_{12}}(11,2)$, and since $\NN_{M_{12}}(11,2)=1$ we
deduce that $X$ is centralised by a Sylow $11$-subgroup of $M_{12}$, and
thus by the whole $Y$, which is a contradiction since in $S$ there are no
elements of order $66$. Finally, a similar argument shows that if $X$ is a
non-trivial element of $\NN_{3^{\cdot}Suz:2}(33,2)$, then $X\cap Suz$ is a
non-trivial element of $\NN_{Suz}(11,2)$. This is impossible since the
maximal subgroups of $Suz$ of order divisible by $11$ are: $\PSU_5(2)$,
$3^5\co M_{11}$ and $M_{12}\co2$, forcing $\NN_{Suz}(11,2)=1$.

$S=Fi_{22}$. Here $|\Out(S)|=2$.
The maximal subgroups of $S$ of order divisible by 13 are isomorphic to
$\tw2F_4(2)$ or $\OO_7(3)$. Since both these  groups have orders not divisible
by $11$, we have that $\NN_{S}(11,13)=1$. Now, $\NN_{\tw2F_4(2)}(13,p)=1$ for
every $p\ne 13$, since the maximal subgroups of $\tw2F_4(2)$ containing a
Sylow $13$-subgroup are $\PSL_2(25)$ and $\PSL_3(3)\co2$ and $C_S(13)=13$.
In $\OO_7(3)$ there are three isomorphism classes of maximal subgroups of
order divisible by 13, namely $G_2(3), \PSL_4(3)\co2$ and $3^{3+3}\co\PSL_3(3)$.
We have that $\NN_{\OO_7(3)}(13,p)=1$ if $p\ne 3$ (and $p\ne 13$), forcing
$\NN_{S}(13,p)=1$ for every prime $p$ different from $3$ and $13$. To deal
with the case $p=3$, we look at the maximal subgroups of $S$ of order
divisible by~11. These are isomorphic to one of the following: $M_{12}$,
$2^{10}\co M_{22}$ and $2\udot \PSU_6(2)$. For any of these groups $Y$ we have
$\NN_{Y}(11,3)=1$, thus the same happens in $S$. Since $|\Out(S)|=2$, we
only need to show that $\NN_{\Aut(S)}(q,2)=1$ for some odd integer $q$. This
is guaranteed by the fact that $\NN_{S}(13,2)=1$ and $C_{\Aut(S)}(13)=13$.

For the last three groups, the Atlas does not contain complete lists of
maximal subgroups, so we need to give a different argument.

$S=Fi_{24}'$. Here $|S|=2^{21}\cdot 3^{16} \cdot 5^2 \cdot 11\cdot 13\cdot 17\cdot 23\cdot 29$, $|\Out(S)|=2$.
Here, a subgroup of order~29 cannot act faithfully on an elementary abelian
$p$-subgroup for $p\ne29$, by the order formula. On the other hand, subgroups
of order~29 are not normalised by elements of order~23.

$S=B$. Here $|S|=2^{41}\cdot 3^{13} \cdot 5^6 \cdot 7^2\cdot 11\cdot 13\cdot 17 \cdot 19\cdot  23\cdot 31 \cdot 47$,
$\Out(S)=1$.
>From the order formula it is clear that a subgroup of order~47 cannot act
non-trivially on an elementary abelian $p$-subgroup of $S$, except
possibly for $p=2$. Since elements of order~47 are self-centralising, and
not normalised by an element of order~31, we must have $p=2$. But the 2-rank of
$S$ is~14 by \cite{La07}, too small for an action of $C_{47}$.

$S=M$. Here $|S|=2^{46}\cdot 3^{20}\cdot 5^9\cdot 7^6\cdot 11^2\cdot 13^2
\cdot 17\cdot 19\cdot 23\cdot 29\cdot 31\cdot 41\cdot 47\cdot 59\cdot 71$,
$\Out(S)=1$.
A subgroup of order~59 cannot act faithfully on an elementary abelian
$p$-subgroup for $p\ne59$, by the order formula. On the other hand, subgroups
of order~59 are not normalised by elements of order~71.
\end{proof}

\subsection{Classical groups of Lie type}
We consider the following setup. Let $S$ be a finite simple group of Lie type.
There exists a simple linear algebraic group $\bH$ of adjoint type defined
over the algebraic closure of a finite field and a Steinberg endomorphism
$F:\bH\rightarrow\bH$ such that the finite group of fixed points $H=\bH^F$
satisfies $S=[H,H]$.

We now make use of the fact that groups of Lie type possess elements of
orders which cannot occur in their Weyl group, and with small centraliser.
These can be found, for example, in the Coxeter tori. For this we need the
existence of \emph{Zsigmondy primitive prime divisors}
(see \cite[Thm.~3.9]{He74}):

\begin{lem}   \label{lem:Zsig}
 Let $q$ be a power of a prime and $e>2$ an integer. Then unless $(q,e)=(2,6)$
 there exists a prime $\ell$ dividing $q^e-1$, but not dividing $q^f-1$ for any
 $f<e$, and $\ell\ge e+1$.
\end{lem}

In Table~\ref{tab:tori} we have collected for each type of classical group
two maximal tori $T_1,T_2$ of $H$ (indicated by their orders). Then the order of
$T_i$ is divisible by a Zsigmondy prime divisor $\ell_i$ of $q^{e_i}-1$,
with $e_i$ given in the table (unless $e_i=2$ or $(e_i,q)=(6,2)$).

\begin{table}[htbp]
\caption{Two tori for classical groups.}  \label{tab:tori}
\[\begin{array}{cc|cc|cc}
 H& & |T_1|& |T_2|& e_1& e_2\cr
\hline
            A_{n-1}& (n\ge2)& (q^n-1)/(q-1)& q^{n-1}-1& n& n-1\cr
  \tw2A_{n-1}& (n\ge3$ odd$)& (q^n+1)/(q+1)& q^{n-1}-1& 2n& n-1\cr
            & (n\ge4$ even$)& q^{n-1}+1& (q^n-1)/(q+1)& 2n-2& n\cr
     B_n,C_n& (n\ge2$ even$)& q^n+1& (q^{n-1}+1)(q+1)& 2n& 2n-2\cr
             & (n\ge3$ odd$)& q^n+1& q^n-1& 2n& n\cr
         D_n& (n\ge4$ even$)& (q^{n-1}+1)(q+1)& (q^{n-1}-1)(q-1)& 2n-2& n-1\cr
                            & (n\ge5$ odd$)& (q^{n-1}+1)(q+1)& q^n-1& 2n-2& n\cr
            \tw2D_n& (n\ge4)& q^n+1& (q^{n-1}+1)(q-1)& 2n& 2n-2\cr
\end{array}\]
\end{table}

\begin{prop}   \label{prop:crosschar}
 Assume that $S$ is of classical Lie type not in characteristic~$p$. Then
 Theorem~\ref{thm:almost-simple} holds for all $S\le G\le\Aut(S)$.
\end{prop}

\begin{proof}
Let $\bH,H$ be as above so that $S=[H,H]$. We distinguish three cases.

\smallskip\noindent
{\bf Case 1:} $A\le S$.\\
The cases when $e_i\le2$, that is, $H$ is of type $A_1$, $A_2$, $\tw2A_2$ or
$B_2$, will be considered in Proposition~\ref{prop:smallrank}. For all other
types, for $X$ we choose a maximal cyclic subgroup of $T_i\cap S$ for $i=1,2$,
with $T_i$ from Table~\ref{tab:tori}. Note that the orders of
$T_1\cap S,T_2\cap S$ are coprime, and $T_i$ is the centraliser in $H$ of any
$s_i\in T_i$ of order $\ell_i$. Assume that $\NN_S^A(X,p)\ne\varnothing$.
By Lemma~\ref{lem:normElem} and the fact that $A\le S$, $X$ normalises a
non-trivial elementary abelian $p$-subgroup $E$ of $S$. Let
$\pi:\tilde\bH\rightarrow\bH$ be a simply-connected covering of $\bH$, and hence
$\ker(\pi)=Z(\tilde\bH)$. We let $\tilde E$ be a (normal) Sylow $p$-subgroup
of the full preimage of $E$ in $\tilde\bH$. Then $\tilde E$ is normalised by
the full preimage of $X$. First assume that $|Z(\tilde\bH)|$ is prime to $p$.
Then $\tilde E\cong E$ is abelian. As $\tilde\bH$ is simply-connected, $p>2$
is not a torsion prime of $\tilde\bH$ (see \cite[Tab.~14.1]{MT}), so an
inductive application of
\cite[Thm.~14.16]{MT} to a sequence of generators of the abelian group
$\tilde E$ shows that $\bC:=C_{\tilde\bH}(\tilde E)$ contains a maximal torus
of $\tilde\bH$ and is connected reductive, hence a subsystem subgroup of
$\tilde\bH$ of maximal rank. Then $N_{\tilde\bH}(\bC)=\bC N_{\tilde\bH}(\bT)$
for any maximal torus $\bT$ of $\tilde\bH$, so $N_{\tilde\bH}(\bC)/\bC$ is
isomorphic to a section of the Weyl group $W$ of $\tilde\bH$. As
$N_{\tilde\bH}(\tilde E)\le N_{\tilde\bH}(C_{\tilde\bH}(\tilde E))
=N_{\tilde\bH}(\bC)$ we see that $N_{\bH}(E)/C_{\bH}(E)$ is a section of $W$.
\par
Now note that the order of the Weyl group of $\bH$ is not divisible by any
prime larger than $e_i$, except for $H$ of type $D_n$, with $n\ge4$ even
and $e_2=n-1$. Here, $\ell_2\ge e_2+1=n$, but $n$ is even so that in fact
$\ell_2>n$ does not divide the order of the Weyl group either. This shows that
elements $s_i\in X$ of order $\ell_i$ must centralise $E$, for $i=1,2$.
So $p$ divides the order of $C_S(s_i)=T_i\cap S$ for $i=1,2$, a contradiction
as these orders are coprime.
\par
The cases when $(q,e_i)=(2,6)$, that is,
$S=\PSL_6(2),\PSL_7(2),\PSU_6(2)$, $\OO_7(2),\OO_8^\pm(2),\OO_9(2)$
will be handled in Proposition~\ref{prop:noZsigmondy}, while
$S=\PSU_4(2)\cong\PSp_4(3)$ will be treated in
Proposition~\ref{prop:smallrank}.

\par
Now assume that $\tilde E$ is non-abelian. Then $p$ divides $|Z(\tilde\bH)|$
and thus $S=\PSL_n(q)$ or $S=\PSU_n(q)$. Let $E_1$ be a minimal non-cyclic
characteristic subgroup of $\tilde E$. Then $E_1$ is of symplectic type,
hence extra-special (see \cite[(23.9)]{Asch}) and normalised by $X$. Write
$|E_1|=p^{2a+1}$, then $p^a\le n$ as $E_1\le\SL_n(q)$ or $\SU_n(q)$. Now the
outer automorphism group of
$E_1$ is $\Sp_{2a}(p)$, and all prime divisors of its order are at most
$(p^a+1)/2<n$. But our Zsigmondy prime divisors $\ell_i$ of $|X|$ satisfy
$\ell_i\ge n$, so again we conclude that $X$ must centralise an element of
order~$p$. We conclude as before.

\smallskip\noindent
{\bf Case 2:} $A\not\le S$ contains diagonal automorphisms.\\
In this case by Proposition~\ref{prop:oddout} we have $S=\PSL_n(q)$ or
$S=\PSU_n(q)$. Here let $X$ be generated by a regular unipotent element. By
Lemma~\ref{lem:out}, if $X$ normalises a non-trivial $p$-subgroup generated by
conjugates of $A$, $X$ must centralise some non-trivial element of order $p$.
But the centraliser of a regular unipotent element in the group $\PGL_n(q)$
resp.~$\PGU_n(q)$ of inner-diagonal automorphisms is obviously unipotent,
hence this case does not occur, as by assumption $p$ is not the defining
characteristic.

\smallskip\noindent
{\bf Case 3:} $A\not\le S$ does not contain diagonal automorphisms.\\
By Proposition~\ref{prop:oddout}, $A$ contains field, graph or graph-field
automorphisms. Now in all cases, a maximal cyclic subgroup $X$ of $T_1\cap S$ can
be identified to a subgroup of the multiplicative group of $\FF_{q^{e_1}}$ by
viewing some isogeny version of $H$ as a classical matrix group. The normaliser
in $S$ of $X$ then acts by field automorphisms of $\FF_{q^{e_1}}/\FF_q$. Using
the embedding into a matrix group one sees that the field automorphisms of
$S$ act on $X$ as the field automorphisms of $\FF_q/\FF_r$, where $r$ is the
characteristic of $\bH$. In particular they induce automorphisms of $X$
different from those induced by $N_S(X)$. So with this choice of $X$ field
automorphisms cannot lead to examples by Lemma~\ref{lem:out}.
Finally, if $S=\OO_8^+(q)$ and $A$ contains graph or graph-field automorphisms
of order~3 then we choose $X$ to be generated by an element $x$ of
order~$(q^2+1)/d$ in a maximal torus $T\le S$ of order~$(q^2+1)^2/d^2$, where
$d=\gcd(q-1,2)$. The normaliser $N_S(T)$ acts by the complex reflection
group $G(4,2,2)$ of 2-power order, while in the extension by a graph or
graph-field automorphism
it acts by the primitive reflection group $G_5$. These automorphisms hence
induce further non-trivial elements normalising $X$, and not centralising $x$.
\end{proof}

We now complete the proof for the small rank cases.

\begin{prop}   \label{prop:smallrank}
 Assume that $S=\PSL_2(q)$ ($q\ge8$), $\PSL_3(q)$, $\PSU_3(q)$ ($q>2$), or
 $\PSp_4(q)$ ($q>2$), and $p\nmid q$. Then Theorem~\ref{thm:almost-simple}
 holds for all $S\le G\le\Aut(S)$.
\end{prop}

\begin{proof}
We just need to deal with the case that $G=S$, since the other possibilities
were already discussed in the proof of Proposition~\ref{prop:crosschar}.
First assume that $S=\PSL_2(q)$. If $q\ge8$ is even, then elements of order
$q+1$ do not normalise any non-trivial $p$-subgroup with $p$ dividing $q-1$,
while elements of order $q-1$ do not normalise any with $p|(q+1)$.
If $q=r^f\ge9$ is odd, elements of order~$r$ do not normalise non-trivial
$p$-subgroups for $2<p|(q^2-1)$. \par
Next let $S=\PSL_3(q)$. Elements of order $(q^2+q+1)/\gcd(3,q-1)$ do not
normalise non-trivial $p$-subgroups for $p$ dividing $q^2-1$, while elements
of order $2$ do not normalise non-trivial $p$-subgroups for $p$ dividing
$(q^2+q+1)/\gcd(3,q-1)$. Similarly for $S=\PSU_3(q)$, $q>2$, we can argue
using elements of order $(q^2-q+1)/\gcd(3,q+1)$, respectively of order~2.
\par
Finally assume that $S=\PSp_4(q)$. Using $X$ of order $q^2+1$ we see that
we must have $p|(q^2+1)$. In this case, take $X$ of order~3.
\end{proof}

\begin{prop}   \label{prop:noZsigmondy}
 Assume that $S$ is one of $\PSL_6(2),\PSL_7(2),\PSU_6(2),\OO_7(2),
 \OO_8^\pm(2)$ or $\OO_9(2)$. Then Theorem~\ref{thm:almost-simple} holds for all
 $S\le G\le\Aut(S)$.
\end{prop}

\begin{table}[htbp]
\caption{Some groups over $\FF_2$.} \label{tab:noZsig}
$\begin{array}{|lll|lll|ll|}\hline
\PSL_6(2)&    (2; 7) && \PSU_6(2)& (11; 3) && \OO_8^+(2)& (7; 5)\\
\PSL_7(2)& (127; 31) && \OO_7(2)&   (7; 5) && \OO_8^-(2)& (17; 3)\\
\OO_9(2)& (17; 5) && & && & \\
     \hline
\end{array}$\\
\end{table}

\begin{proof}
In all of these groups, just one of the two Zsigmondy primes $\ell_i$ exists.
By the argument given in the proof of Proposition~\ref{prop:crosschar}, we
still obtain that either elements of order $\ell_i$ centralise a $p$-element
or that $S\ne G$. We may then
conclude as in the proof of Proposition~\ref{prop:Spor}, using an additional
prime as in Table~\ref{tab:noZsig}, except in the two cases when $|\Out(S)|>2$:

Let $S=\PSU_6(2)$. Here $\ell_1=11$ shows that $p=11$ if $G=S$, and since a
subgroup of order $7$ does not normalise one of order $11$, we reach a
contradiction in this case. We assume therefore that $p=3$ and
$A\not\leq S$. Let first $X$ be a subgroup of $S$ of order $11$ and $Y$ a
maximal subgroup of $S$ containing $X$. Then $Y\simeq \PSU_5(2)$ or $M_{22}$,
and therefore $\NN_Y(X,3)=\NN_S(X,3)=1$. Now if $E\in\NN^A_G(X,3)$ we have
that $E\cap S\in\NN_S(X,3)=1 $ and so $E=C_E(X)$ has order $3$, and, $E$ being
generated by conjugates of $A$, we have that $E$ is a conjugate of $A$.
Now the group $S$ has four classes of outer elements of order three, denoted
$3D, 3E, 3F$ and $3G$ in \cite{Atlas}. Amongst these just $3D$ has centraliser
in $S$ divisible by $11$, namely $C_S(3D)\simeq \PSU_5(2)$. We have therefore
that $A=\seq{x}$ for some $x$ in $3D$. Now take $X$ a subgroup
of $S$ of order $7$. We may argue as before. Since a maximal subgroup $Y$
of $S$ containing $X$ is isomorphic to one of
$$M_{22},\, 2^9.\PSL_3(4),\, \PSU_4(3)\co2,\, \PSp_6(2),\, \PSL_3(4)\co2,$$
we have that $\NN_Y(X,3)=\NN_S(X,3)=1$, and therefore any element
$E\in \NN^A_G(X,3)$ is a cyclic subgroup conjugate to $A$ and centralised
by $X$. This is a contradiction since $7$
does not divide $|\PSU_5(2)|$.

Let $S=\OO_8^+(2)$. The prime $\ell_2=7$ shows that $p=7$ if $A\le S$.
Since a subgroup of order $5$ does not normalise any non-trivial $7$-subgroup,
we reach a contradiction if $A\leq S$. Let $p=3$ and $A\not\leq S$.
In $G$ there are three classes of outer $3$-elements, two of order $3$ and one
of order $9$. In all cases $5$ does not divide the order of their centralisers
in $S$. Thus if $X$ is a cyclic subgroup of order $5$ we reach a contradiction
with Lemma~\ref{lem:out}.
\end{proof}

\subsection{Groups of exceptional type}
In this section we prove Theorem~\ref{thm:almost-simple} when $S$ is one of
the exceptional groups of Lie type. We keep the setting from the beginning of
the previous subsection. Note that we need not treat $\tw2B_2(2)$ (which is
solvable), $G_2(2)\simeq\PSU_3(3).2$, $^2G_2(3)\simeq\PSL_2(8).3$ and
$\tw2F_4(2)'$ (see Section~\ref{subsec:spor}).

As in the case of classical groups we provide in Table~\ref{tab:exc} for each
type of group two maximal tori of $H$, indicated by their orders. Here, we
denote by $\Phi_n$ the $n$-th rational cyclotomic polynomial evaluated at $q$,
and moreover we let $\Phi_8'=q^2+\sqrt{2}q+1,\, \Phi_8''=q^2-\sqrt{2}q+1,
\Phi_{24}'=q^4+\sqrt{2}q^3+q^2+\sqrt{2}q+1$,
$\Phi_{24}''=q^4-\sqrt{2}q^3+q^2-\sqrt{2}q+1$ for $q^2=2^{2f+1}$, and
$\Phi_{12}'=q^2+\sqrt{3}q+1,\, \Phi_{12}''=q^2-\sqrt{3}q+1$ for $q^2=3^{2f+1}$.
We then have
$\gcd(|T_1|,|T_2|)=d$, where $d=(3,q-1)$,  $(3,q+1)$ and $d=(2,q-1)$ for
$S=E_6(q), \tw2E_6(q), E_7(q)$ respectively, and $d=1$ otherwise. Furthermore,
$|H:S|=|T_i:T_i\cap S|=d$ in all cases.

\begin{table}[htbp]
\caption{Two tori for exceptional groups.}  \label{tab:exc}
\[\begin{array}{|cc|ccc||cc|ccc|}
\hline
 H& & |T_1|& |T_2|&& H& & |T_1|& |T_2|& d\cr
\hline
\tw2B_2(q^2)& (q^2\ge8)& \Phi_8'& \Phi_8''&& F_4(q)& & \Phi_8& \Phi_{12}& 1\cr
  ^2G_2(q^2)& (q^2\ge27)& \Phi_{12}'& \Phi_{12}''&& E_6(q)& & \Phi_3\Phi_{12}& \Phi_9& (3,q-1)\cr
      G_2(q)& (q\ge3)& \Phi_3& \Phi_6&& \tw2E_6(q)& & \Phi_6\Phi_{12}& \Phi_{18}& (3,q+1)\cr
  \tw3D_4(q)& & \Phi_3^2& \Phi_{12}&& E_7(q)& & \Phi_2\Phi_{14}& \Phi_1\Phi_7& (2,q-1)\cr
\tw2F_4(q^2)& (q^2\ge8)& \Phi_{24}'& \Phi_{24}''&& E_8(q)& & \Phi_{15}& \Phi_{30}& 1\cr
\hline
\end{array}\]
\end{table}

\begin{prop}   \label{prop:exc}
 Assume that $S$ is of exceptional Lie type not in characteristic~$p$. Then
 Theorem~\ref{thm:almost-simple} holds for all $S\le G\le\Aut(S)$.
\end{prop}

\begin{proof}
First assume that $A\le S$.
Using for $X$ maximal cyclic subgroups of $S\cap T_i$, for $T_i$ as listed in
Table~\ref{tab:exc}, we conclude by the same arguments as in the proof of
Proposition~\ref{prop:crosschar} that in a possible counterexample to
Theorem~\ref{thm:almost-simple} the prime $p$ would divide the orders of both
$S\cap T_i$, $i=1,2$, which is a contradiction as their orders are coprime,
unless
possibly if $p$ is a torsion prime for $H$. The torsion primes for groups of
exceptional Lie type are just the bad primes (see \cite[Tab.~14.1]{MT}); in
particular $p\le5$, and even $p=3$ unless $S=E_8(q)$. The maximal rank of an
elementary abelian $p$-subgroup of $H$, for $p$ odd, is at most the rank
$m$ of $H$, see e.g.~\cite[Thm.~4.10.3]{GLS}. It is easy to check that for
all bad primes $p\nmid q$ and $s\le m$, $p^s-1$ is not divisible by $\ell_i$
for $i\in\{1,2\}$, so again $T_i$ must centralise a non-trivial $p$-element,
which contradicts the fact that $\gcd(|T_1|,|T_2|)=1$, except for the groups
$F_4(2),E_6(2)$ and $\tw2E_6(2)$ (each with $p=3$). In all of the
latter cases, at least one of the $\ell_i$ does not divide $3^s-1$ for
$s\le 4$, and does not divide the centraliser order of an element of order $3$
either, so again we are done.
\par
Now assume that $A\not\le S$. Then either $A$ contains field automorphisms, in
which case taking for $X$ a maximal cyclic subgroup of $T_1$ shows that no
example arises by
Lemma~\ref{lem:out} as field automorphisms induce proper non-inner
automorphisms on this torus. Or, we have that $S=E_6(q)$ or $\tw2E_6(q)$,
$p=3$, and $A$ contains diagonal automorphisms. In this case we take for $X$
the subgroup generated by a regular unipotent element; this has a unipotent
centraliser in the group of inner-diagonal automorphisms and thus we are done.
\end{proof}

\subsection{Groups of Lie type in defining characteristic}
If $p$ is the defining prime for $S$, we can again make use of the two tori
$T_1,T_2$ introduced before.

\begin{prop}   \label{prop:defchar}
 Assume that $S$ is of Lie type in characteristic~$p$. Then
 Theorem~\ref{thm:almost-simple} holds for all $S\le G\le\Aut(S)$.
\end{prop}

\begin{proof}
Let $A\le H$ be cyclic of order~$p$ and $1\ne P\le H$ be a $p$-subgroup
generated
by conjugates of $A$. First assume that $A\le S$. Then $P$ is a non-trivial
unipotent subgroup of $S$, hence its normaliser $N_S(P)$ is contained in some
proper parabolic subgroup of $S$ (see \cite[Thm.~26.5]{MT}). Let $s$ be a
regular semisimple element of $S$ in the torus $T_1$ as given in
Table~\ref{tab:tori} when $S$ is classical, or in Table~\ref{tab:exc} in
case $S$ is exceptional. Then the centraliser $C_S(s)$ is contained in $T_1$,
in particular $s$ does not centralise any non-trivial split torus of $\bH$ and
so is not contained in a proper parabolic subgroup of $S$. Thus
$\NN_S^A(X,p)=\varnothing$.
\par
If $A\not\le S$, then by Proposition~\ref{prop:oddout} either $A$ contains a
field automorphism of $S$, or $p=3$ and $A$ contains a graph or graph-field
automorphism. According to Lemma~\ref{lem:out}, $X$ is centralised by an outer
$p$-element. Now as pointed out in the proof of
Propositions~\ref{prop:crosschar} and~\ref{prop:exc}, field automorphisms do
not enlarge the centraliser of $X$ as defined above, so we may assume
that $S=\OO_8^+(q)$, $p=3$ and $H$ involves a graph or graph-field
automorphism. In this case take $X$ generated by a semisimple element of
order $(q^2+1)/\gcd(2,q-1)$ and conclude as in the proof of
Proposition~\ref{prop:crosschar}.
\end{proof}

\section{Proof of Theorem A}   \label{sec:main}
In this section we complete the proof of Theorem A.
We need the following result, whose proof can be found in \cite[Thm.~4.2]{INW}.

\begin{lem}    \label{Lem_W2_sol}
 Let $G$ be a finite group and $V$ a faithful irreducible $G$-module. Assume
 that $p$ is an odd prime number different from the characteristic of $V$ and
 that $A$ is a subgroup of $G$ of order $p$ that lies in $O_p(G)$.
 Then there exists an element $v\in V$ such that
 $$A\not\subseteq \bigcup_{g\in G} C_G(v)^g.$$
\end{lem}

Given a finite group $G$ and a subgroup $A\leq G$ we say that the pair
$(G,A)$ satisfies $(\ast)$ if for every conjugacy class $C$ of $G$ there exists
$g\in C$ such that $A$ is subnormal in $\seq{A,g}$.

\begin{proof}[Proof of Theorem~A]
We argue by contradiction: assume that $G$ is a finite group, $A$ an odd
$p$-subgroup of $G$, $A\not\leq O_p(G)$, and the pair $(G,A)$
satisfies the condition  $(\ast)$. Moreover we assume that that $|G|+|A|$ is
minimal with respect to these conditions. We proceed by steps.

\smallskip\noindent
{\bf Step 1.} We have $O_p(G)=1$.

Indeed, note that $(G/O_p(G),AO_p(G)/O_p(G))$ satisfies $(\ast)$, therefore if
$O_p(G)\ne 1$ by our minimal assumption, we would have that
$AO_p(G)/O_p(G)\leq O_p(G/O_p(G))=1$, which is a contradiction.

\smallskip\noindent
{\bf Step 2.} We have $|A|=p$.

Let $B$ be a proper subgroup of $A$ and note that $(G,B)$ satisfies $(\ast)$.
By the minimal choice, every proper non-trivial
subgroup of $A$ lies in $O_p(G)$.
Since $O_p(G)=1$, we conclude that $B=1$, i.e., $A$ has order $p$.\\
>From now on we set $A=\seq{a}$.

\smallskip\noindent
{\bf Step 3.} $G$ has a unique minimal normal subgroup $M$.

Assume that $M$ and $N$ are two distinct minimal normal subgroups of $G$ and
assume also that $A\not\leq N$. Then $(G/N,AN/N)$ satisfies $(\ast)$ and so,
by our minimal choice we have that $AN/N\leq O_p(G/N)$. In particular,
$AN\su G$. Then also $AN\cap M\su G$. If $A\leq M$, then
$A=A(N\cap M)=AN\cap M$, and we have that $A\su G$. Since $A$ is
a $p$-subgroup, then $A\leq O_p(G)$, which is not the case.
Therefore $A\not\leq M$ and by the same arguments as for $N$ above, we conclude
that $AM$, and thus also $AM\cap AN$, is subnormal in $G$. Finally note that
$M\cap AN=1$. Indeed, otherwise we have that $A\leq MN$, as $|A|=p$ and
$M\cap N=1$. Now $MN/N$ is a minimal normal subgroup of $G/N$ and, being
isomorphic to $M$, it is not a $p$-subgroup by Step~1. Then
$MN/N\cap O_p(G/N)= 1$, forcing $AN/N=1$ a contradiction. Thus $M\cap AN=1$ and
$A=A(M\cap AN)=AM\cap AN$, therefore is subnormal in $G$, which again
contradicts Step 1.

\smallskip\noindent
{\bf Step 4.} $M$ is non-abelian.

Assume that $M$ is an elementary abelian $q$-group, with $q$ a prime different
from $p$, by Step~1. Let $Y/M=O_p(G/M)$, then by our minimal assumption
$A\leq Y$. We take $P\in\Syl_p(Y)$ such that $A\leq P$. By the Frattini
argument $G=YN=MN$, with $N=N_G(P)$. Now $[N_M(P),P]\leq M\cap P=1$, thus
$N_M(P)=C_M(P)$. Also, $M$ being normal and abelian, $C_M(P)$ is
normalised by both $M$ and $N_G(P)$, thus $C_M(P)\trianglelefteq G$. As $M$ is
the unique
minimal normal subgroup of $G$, we have that either $C_M(P)=1$ or $C_M(P)=M$.
Note that in the latter case $P$ is normal in $G$, which contradicts Step 1.
Therefore we have that $G$ is a split extension $G=M\rtimes N$.
Moreover, since $M$ is the unique minimal normal subgroup of $G$, $C_N(M)=1$,
i.e., $N$ acts faithfully on $M$. Let $m$ be an arbitrary non-trivial element of $M$. By
condition $(\ast)$ there exists some $n\in N$ such that $A\su\seq{A,m^n}$.
In particular the subgroup $V:=\seq{a,a^{m^n}}$ is a $p$-group. As $m^n\in M$,
$MV=MA$, and therefore
$$V=MA\cap V=(M\cap V)A=A,$$
as $M$ and $V$ have coprime orders. Therefore $\seq{a}=\seq{a^{m^n}}$, i.e.,
$m^n$ normalises $A$. In particular, as $M$ is a normal $q$-subgroup, we have
that
$$[a, m^n]\in A\cap M=1,$$
which means that $A\subseteq C_N(m)^n$. By the arbitrary choice of $m$ in $M$ we
have reached a contradiction with Lemma \ref{Lem_W2_sol}.

\smallskip\noindent
{\bf Step 5.}

Let $M=S_1\times S_2\times \ldots \times S_n$ be the unique minimal normal
subgroup of $G$, with all the $S_i$'s isomorphic to a finite non-abelian
simple group $S$. Denote by $\pi_i$ the projection map of $M$ onto $S_i$,
for every $i=1,2,\ldots n$. Let also $1=x_1, x_2,\ldots ,x_n$ be elements of
$G$ such that $S_1^{x_i}=S_i$, for $i=1,2,\ldots, n$.
Let $K$ be the kernel of the permutation action of $G$ on the set
$\cS:=\set{S_1,S_2,\ldots , S_n}$, i.e.,
$$K:=\bigcap_{i=1}^n N_G(S_i).$$
We treat separately the two cases: $A\not\leq K$ and $A\leq K$.

\smallskip\noindent
{\bf{Case 1. $A\not\leq K$.}}

\noindent
Set $\overline{G}:=G/K$ and use the ``bar'' notation to denote subgroups and
elements of $\overline{G}$. By induction, we have that
$\overline{A}\leq O_p(\overline{G})$. Since $p$ is odd, by Gluck's Theorem
(\cite[Cor.~5.7]{MW}) there exists a proper non-empty subset $\cR\subset \cS$
such that
$$\overline{G}_\cR\cap O_p(\overline{G})=\overline{1},$$
where $G_\cR$ denotes the stabiliser in $G$ of the set $\cR$. Without loss of
generality, we may assume $\cR=\set{S_1,\ldots,S_r}$ for some $r< n$.
Let $q$ be any prime different from $p$ dividing $|S_1|$, and let $s_1\in S_1$
be any non-trivial $q$-element. Set
$$s_\cR:=s_1s_1^{x_2}\ldots s_1^{x_r}\in M.$$
By assumption, there exists a $G$-conjugate of $s_\cR$, say $y:=s_\cR^g$,
such that $A$ is subnormal in $\seq{a,y}$, in particular $\seq{a,a^y}$ is
a $p$-subgroup. Thus $[a,y]=a^{-1}a^y$ is a $p$-element.
Also, $[a,y]$ is $G$-conjugate to $[a^{g^{-1}},s_\cR]$.
Since $\overline{a^{g^{-1}}}$ is a non-trivial element of $O_p(\overline{G})$,
$a^{g^{-1}}$ does not stabilise $\cR$, therefore there exists some $i\in \cR$
such that $(S_i)^{a^{g^{-1}}}=S_j$ for some $j\not\in \cR$, this forces that
$\pi_j([a^{g^{-1}},s_\cR])$ is a non-trivial $q$-element of $S_j$, and since
$p\ne q$, $[a^{g^{-1}},s_\cR]$ cannot be a non-trivial $p$-element of $M$. So
we have $[a,y]=1$, but then $\seq{a}$ stabilises $\cR$, which is in
contradiction with $\overline{G}_\cR\cap O_p(\overline{G})=\overline{1}$.

\smallskip\noindent
{\bf{Case 2. $A\leq K$.}}

\noindent
We consider first the case in which $A\leq C_G(S_i)$, for every $i=1,\ldots,n$.
Then
$$A\leq \bigcap_{i=1}^n C_G(S_i)=(C_G(S_1))_G,$$
the normal core of $S_1$ in $G$. Since $M$ is the unique minimal normal
subgroup of $G$ and $M\not \leq  (C_G(S_1))_G$, we necessarily have that
$(C_G(S_1))_G=1$, and so $A=1$, which is a contradiction.

\noindent
Assume now that $A$ does not centralise some $S_i$, say $S_1$.
Let $1\ne s_1\in S_1$ and let $m=s_1^{\null}s_1^{x_2}\ldots s_1^{x_n}\in M$.
Let $g\in G$ be such that $A\su\seq{A,m^g}$. Writing
$m^g=h_1k$, with $h_1=s_1^{x_ig}\in S_1$ for some $i=1,\ldots,n$,
and $k=\prod_{j\ne i}s_1^{x_jg}\in S_2\times\ldots \times S_n$ we have that
for every $u,v\in \mathbb{N}$
$$[a^u,(m^g)^v]=[a^u,h_1^vk^v]=[a^u,k^v][a^u,h_1^v]^{k^v}
                          =[a^u,h_1^v][a^u,k^v],$$
since $A$ normalises each $S_i$ and $S_1$ is centralised by $S_j$, for every
$j\ne 1$. In particular we have that
$$[A,\seq{m^g}]= [A,\seq{h_1}]\times [A,\seq{k}].$$
Therefore $\pi_1([A,\seq{m^g}])=[A,\seq{h_1}]$ is a $p$-subgroup
of $S_1$. Finally note that $h_1=s_1^{x_ig}$, and so for the arbitrary element
$s_1\in S_1$ there exists $x_ig\in N_G(S_1)$ such that
$A\su\seq{A, s_1^{x_ig}}$. In particular if $s_1$ is chosen to be a
$p'$-element of $S$ we have that $s_1$ normalises a non-trivial
$p$-subgroup of $G$ which is generated by $G$-conjugates of $A$. Therefore,
as $A\not\leq C_G(S_1)$, we have proved that the almost simple group
$\tilde G:=N_G(S_1)/C_G(S_1)$ contains a non-trivial subgroup of
order~$p$, namely $\tilde A:=AC_G(S_1)/C_G(S_1)$, such that for every
cyclic $p'$-subgroup $\widetilde{X}$ of $F^*(\tilde G)$ we have that
$\NN_{\tilde G}^{\tilde A}(\tilde X,p)\ne \varnothing$. This is
in contradiction to Theorem~\ref{thm:almost-simple}.
\end{proof}

We end this section by showing with an easy example that for $p=2$ Theorem A
is no more true.

\begin{exmp}
Let $H$ be a Sylow $2$-subgroup of $\GL_2(3)$, namely $H$ is a semidihedral
group of order 16, acting, in the natural way, on the natural module
$M\simeq C_3\times C_3$. Let $G$ be the semidirect product $M\rtimes H$, and
$a$ a non-central involution of $H$. Since $O_2(G)=1$, $\seq{a}$ is not
subnormal in $G$. We show that $\seq{a}$ satisfies $(\ast)$.
A non-trivial element of $G$ has order either a $2$-power, or 3, or 6. In the
first case, it is conjugate to an element $g$ of $H$ and so $\seq{a}$ is
subnormal in the $2$-group $\seq{a,g}$.
In the second case, the element lies in $M$, but note that
every element of $M$ centralises a conjugate in $H$ of $a$, i.e.,
$M=\bigcup_{x\in H}C_M(a^x)$, thus there exists an $x\in H$ such that
$\seq{a}$ is subnormal in $\seq{a,g^x}\simeq C_6$. In the latter case, up
to conjugation, we have $\seq{a,g}=\seq{g}$.
\end{exmp}

\section{Other conditions for subnormality}   \label{sec:others}
As stated in the Introduction, in this Section we briefly analyse
similar variations related to the other criteria for subnormality given by the
original Theorem of Wielandt (namely conditions (iii) and (iv)). We see that in
general these generalisations fail to guarantee the subnormality of odd $p$-
subgroups.

Given a finite group $G$ and an odd $p$-subgroup $A$ of $G$, we consider the
following condition:
\begin{itemize}
\item[($\ast \ast$)] \textit{for every conjugacy class $C$ of $G$ there exists
  $g\in C$ such that $A\su\seq{A,A^g}$}.
\end{itemize}

\noindent
It is trivial that condition $(\ast)$ implies $(\ast \ast)$.\\
The next result shows that $(\ast\ast)$ is enough to guarantee the
subnormality of $A$ in the class of finite solvable groups.

\begin{thm}   \label{thm:one}
 Let $G$ be a finite solvable group and $p$ a prime. If there exists a
 $p$-subgroup $A$ of $G$ satisfying $(\ast \ast)$ then $A\leq O_p(G)$.
\end{thm}

\begin{proof}
As $A$ is nilpotent, every subgroup of $A$ also satisfies $(\ast \ast)$.
Therefore we can assume that $A$ is cyclic, say $A=\seq{a}$.\\
We argue by induction on the order of $G$. \\
Note that if $M$ is a minimal normal subgroup of $G$ the assumption holds
for the group $G/M$. Thus in particular, we may assume
that $G$ admits a unique minimal normal subgroup, say $M$, and that
$aM\in O_p(G/M):=Y/M$. Now if $M$ is a $p$-group we are done. Let $M$ be an
elementary abelian $q$-group, with $q\ne p$. Take $P$ a Sylow
$p$-subgroup of $Y$ containing $a$, so that by
the Frattini argument $G=YN=MN$, with $N=N_G(P)$.
Being $M$ minimal normal in $G$, we have that
$C_M(P)=M\cap N_G(P)=1$ (otherwise $Y=M\times P$ and $a\in O_p(G)$). Thus
$G=M\rtimes N$. Since also $M$ is the unique minimal normal subgroup of $G$ we
have $C_N(M)=1$. Let $m$ be a non-trivial element of $M$. By assumption there
exists $n\in N$ such that the subgroup $V:=\seq{a,a^{m^n}}$ is nilpotent. In
particular, as $m^n\in M$, $MV=M\seq{a}$, and therefore
$$V=M\seq{a}\cap V=(M\cap V)\times \seq{a},$$
forcing $\seq{a}=\seq{a^{m^n}}$. So
$$[a, m^n]\in A\cap M=1,$$
which means that $A\subseteq C_N(m)^n$. By the arbitrary choice of $m$ in $M$
we have reached a contradiction with Lemma~\ref{Lem_W2_sol}.
\end{proof}

For non-solvable groups the situation is completely different and the following
example shows that there are almost simple groups with non-trivial
$p$-subgroups satisfying $(**)$.

\begin{exmp}
Every subgroup of $\fS_8$ generated by a $3$-cycle satisfies $(\ast\ast)$.
Indeed, let $A=\seq{(123)}$ and $C$ a conjugacy class of $\fS_8$.
If every element of $C$ is the product of at least three disjoint cycles of
length $>1$, then $C$ contains $g=(14...)(25...)(36...)...$ and so $A$ is
subnormal in the abelian subgroup $\seq{A,A^g}$.
If $C$ contains a $k$-cycle, then $k\geq 6$ otherwise there exists $g$ in $C$
fixing pointwise $\{1,2,3\}$ and so $\seq{A,A^g}=A$. But then take
$g=(142536...)\in C$ and argue as before. The remaining case is when the
elements of $C$ are products of two disjoint cycles and fix at most two points.
If one of these cycle is a $3$-cycle, then $g=(123)...\in C$, forcing again
$A^g=A$. We can then assume that one cycle is at least a $2$-cycle and the other
a $4$-cycle, but then take $g=(14...)(2536...)\in C$ and conclude as before.

A similar behaviour can be noticed for every prime $p$, if $n$ is big
enough.
\end{exmp}

\noindent
{\bf Acknowledgement.} The first author expresses his gratitude to C. Casolo
for proposing this topic and for helpful discussions.



\begin{thebibliography}{13}

\bibitem{Asch}
{\sc M. Aschbacher}, \emph{Finite Group Theory}. Cambridge Studies in
  Advanced Mathematics, 10. Cambridge University Press, Cambridge, 1986.

\bibitem{AK}
{\sc M. Aschbacher, P. B. Kleidman}, On a conjecture of Quillen and a lemma of
  Robinson. \emph{Arch. Math. (Basel) \bf55} (1990), 209--217.

\bibitem{CC}
{\sc C. Casolo}, A criterion for subnormality and Wielandt complexes in
  finite groups. \emph{J. Algebra \bf169} (1994), 605--624.

\bibitem{Atlas}
{\sc J. H. Conway, R. T. Curtis, S. P. Norton, R. A. Parker, R. A. Wilson},
  \emph{Atlas of Finite Groups}. Oxford University Press, Eynsham, 1985.

\bibitem{GLS}
{\sc D.~Gorenstein, R.~Lyons, R.~Solomon}, \emph{The Classification of
  the Finite Simple Groups. Number 3}. American Mathematical Society,
  Providence, RI, 1998.

\bibitem{G1}
{\sc S.~Guest}, A solvable version of the Baer-Suzuki theorem.
  \emph{Trans. Amer. Math. Soc. \bf 362} (2010), 5909--5946.

\bibitem{G2}
{\sc S.~Guest}, Further solvable analogues of the Baer--Suzuki theorem and
  generation of nonsolvable groups. \emph{J. Algebra \bf 360} (2012), 92--114.

\bibitem{GM1}
{\sc R.M.~Guralnick, G.~Malle}, Variations on the Baer-Suzuki theorem.
  \emph{Math. Z. \bf 279} (2015), no. 3-4, 981--1006.

\bibitem{GM2}
{\sc R.M.~Guralnick, G.~Malle}, Products and commutators of classes in
  algebraic groups. \emph{Math. Ann. \bf 362} (2015), 743--771.

\bibitem{GR}
{\sc R.M.~Guralnick, G.R.~Robinson}, On extensions of the Baer-Suzuki theorem.
  \emph{Israel J. Math. \bf82} (1993), 281--297.

\bibitem{He74}
{\sc C. Hering}, Transitive linear groups and linear groups which contain
  irreducible subgroups of prime order. \emph{Geom. Dedicata \bf2} (1974),
  425--460.

\bibitem{INW}
{\sc I.M. Isaacs, G. Navarro, T.R. Wolf}, Finite group elements where no
  irreducible character vanishes. \emph{J. Algebra \bf 222} (1999), 413--423.

\bibitem{La07}
{\sc R. Lawther}, $2F$-modules, abelian sets of roots and 2-ranks.
  \emph{J. Algebra \bf 307} (2007), 614--642.

\bibitem{LS}
{\sc J.~Lennox, S.E.~Stonehewer}, \emph{Subnormal Subgroups of Groups}.
  Oxford Mathematical Monographs, Oxford University Press, 1987.

\bibitem{MT}
{\sc G. Malle, D. Testerman}, \emph{Linear Algebraic Groups and Finite Groups
  of Lie Type}. Cambridge Studies in Advanced Mathematics, 133.
  Cambridge University Press, Cambridge, 2011.

\bibitem{MW}
{\sc O.~Manz, T.R.~Wolf}, \emph{Representations of Solvable Groups}.
  London Mathematical Society Lecture Note Series, 185. Cambridge University
  Press, Cambridge, 1993.

\bibitem{X}
{\sc W.J.~Xiao}, Glauberman's conjecture, Mazurov's problem and Peng's problem.
  \emph{Sci. China Ser. A \bf 34} (1991), no. 9, 1025--1031.

\end{thebibliography}
\end{document}